\documentclass[11pt]{article}
\usepackage{amsmath,amsthm,amssymb,hyperref, color,fullpage, bm}
\usepackage{blkarray,bigstrut}
\usepackage[T1]{fontenc}
\newcommand{\remove}[1]{}

\makeatletter
\newtheorem*{rep@theorem}{\rep@title}
\newcommand{\newreptheorem}[2]{%
\newenvironment{rep#1}[1]{%
 \def\rep@title{#2 \ref{##1}}%
 \begin{rep@theorem}}%
 {\end{rep@theorem}}}
\makeatother

\hypersetup{
	colorlinks=true,
	linkcolor=blue,
        urlcolor  = blue,
	citecolor=blue,
	linktoc=page
}

\newtheorem{thm}{Theorem}[section]
\newreptheorem{thm}{Theorem}
\newtheorem{claim}[thm]{Claim}
\newtheorem{lem}[thm]{Lemma}
\newreptheorem{lem}{Lemma}
\newtheorem{cor}[thm]{Corollary}

\newtheorem{conjecture}[thm]{Conjecture}

\newtheorem{question}{Question}

\newtheorem{prop}[thm]{Proposition}
\newtheorem{notation}[thm]{Notation}

\theoremstyle{definition}
\newtheorem{define}[thm]{Definition}
\newtheorem{rem}[thm]{Remark}

\def\R{{\mathbb{R}}}

\def\cP{{\cal P}}

\def\cM{{\cal M}}

\def\cH{{\cal H}}

\def\cS{{\mathcal S}}

\newcommand{\Span}{\operatorname{span}}

\newcommand{\eps}{\epsilon}

\begin{document}

\title{Spread Furstenberg Sets}

\author{Paige Bright\footnote{Department of Mathematics, University of British Columbia. Email \href{paigeb@math.ubc.ca}{\nolinkurl{paigeb@math.ubc.ca}}.} and Manik Dhar\footnote{Department of Mathematics, Massachusetts Institute of Techonology. Email: \href{dmanik@mit.edu}{\nolinkurl{dmanik@mit.edu}}. Part of this work was done at Princeton University where the author was supported by NSF grant DMS-1953807.}}

\date{}

\maketitle

\begin{abstract}
We obtain new bounds for (a variant of) the Furstenberg set problem for high dimensional flats over $\R^n$. In particular, let $F\subset \R^n$, $1\leq k \leq n-1$, $s\in (0,k]$, and $t\in (0,k(n-k)]$. We say that $F$ is a $(s,t;k)$-spread Furstenberg set if there exists a $t$-dimensional set of subspaces $\mathcal P \subset \mathcal G(n,k)$ such that for all $P\in \mathcal P$, there exists a translation vector $a_P \in \R^n$ such that $\dim(F\cap (P + a_P)) \geq s$. We show that given $k \geq k_0 +1$  (where $k_0:= k_0(n)$ is sufficiently large) and $s>k_0$, every $(s,t;k)$-spread Furstenberg set $F$ in $\R^n$ satisfies
\[
\dim F \geq n-k + s - \frac{k(n-k) - t}{\lceil s\rceil -k_0 +1}.
\]
Our methodology is motivated by the work of the second author, Dvir, and Lund over finite fields.
\end{abstract} 



\section{Introduction}

In this paper, we study a natural variant of the Furstenberg set problem for high dimensional flats over $\R^n$. To put this variant (and similarly our results) into context, we begin with some background on Furstenberg sets: a fractal analogue of Kakeya/Besicovitch sets.

\subsection{The Kakeya Conjecture and Furstenberg Sets Associated to Lines}

A \textit{Kakeya set}, sometimes referred to as a \textit{Besicovitch set}, is a compact set $K\subset \R^n$ that contains a unit line segment in every direction. Due to work of Besicovitch, it is known that there exists Kakeya sets of Lebesgue measure zero for $n\geq 2$. However, the Kakeya conjecture suggests that every Kakeya set $K\subset \R^n$ has $\dim K =n$, where here and throughout the paper $\dim$ refers to Hausdorff dimension. In $\R^2$, the conjecture was resolved \cite{Davies_1971,cordobakakeya}, and the problem remains open for $n\geq 3$ (though much progress has been made, see \cite{Wolff99, Katz2002} for surveys of progress on the Kakeya conjecture).

Analogously, the Furstenberg set problem (associated to lines) asks how small the Hausdorff dimension of an $(s,t)$-Furstenberg set can be. 

\begin{define}[$(s,t)$-Furstenberg sets]
Let $s\in (0,1]$ and let $t\in (0,2(n-1)]$. We say a Borel set $F\subset \R^n$ is an $(s,t)$\textit{-Furstenberg set} associated to lines, or simply an $(s,t)$\textit{-Furstenberg set}, if there exists a non-empty set of (affine) lines $\mathcal L$ such that $\dim \mathcal L \geq t$ and $\dim (F\cap \ell) \geq s$ for all $\ell \in \mathcal L$.
\end{define}

\begin{rem}
Note that the set of affine $k$-flats in $\R^n$ (denoted $\mathcal A(n,k)$ for the affine Grassmannian) is a metric space with a corresponding notion of Hausdorff dimension. We postpone discussion of this until Section \ref{sec:preliminaries}. For our purposes, we will only study measurable Furstenberg sets.
\end{rem}

\noindent Hence, a Kakeya set $K\subset \R^n$ is a $(1,n-1)$-Furstenberg set, as there exists a family of lines $\mathcal L$ with $\dim \mathcal L \geq n-1$ with $\dim (K\cap \ell) = 1$; simply take the set of lines to be lines containing line segments within $K$. However, note that the associated set of lines for an arbitrary $(s,t)$-Furstenberg set \textit{need not} be spread out in different directions (e.g. you can have a $t$-dimensional set of parallel lines).

Progress towards finding sharp bounds for the Furstenberg set problem has rapidly developed in the 25 years since Wolff initially wrote about the problem \cite{Wolff99}. Detailed accounts of the development of this problem can be found within the work of Orponen--Shmerkin \cite{orponenshmerkinabc} and Ren--Wang \cite{renwang} who collectively resolved the Furstenberg set problem in the plane last year. They show that given an $(s,t)$-Furstenberg set $F\subset \R^2$,
\[
\dim F \geq \min\{s + t, \frac{3s+t}{2}, s+1\}.
\]
It is known that this lower bound is sharp. Notably, Wolff showed the a sharp example that obtains the lower bound of $\frac{3s+t}{2}$ corresponds to a sharp example in the discrete setting for the Szemer\'edi--Trotter theorem regarding point-line incidences. In fact, both the Kakeya conjecture and the Furstenberg set problem have seen a lot of development in the discrete setting, especially over finite fields.

Let $\mathbb{F}_q$ be a finite field with $q$ elements. In the finite field setting, a Kakeya set is a set $K\subset \mathbb{F}_q^n$ that contains a line in every direction. I.e., for every $a\in \mathbb{F}_q^n \setminus \{0\}$, there exists a $b\in \mathbb{F}_q^n$ such that $\{at + b : t\in \mathbb{F}_q\} \subset K$. Akin to the Kakeya conjecture over $\R^n$, it was conjectured that Kakeya sets in $\mathbb{F}_q^n$ should also be correspondingly large (e.g. of cardinality $\approx |\mathbb{F}_q^n| = q^n$ up to some multiplicative constant). However, in contrast to the Euclidean setting, Dvir \textit{resolved} this conjecture in finite fields in all dimensions using the polynomial method, showing that every Kakeya set $K\subset \mathbb{F}_q^n$ has cardinality $\gtrsim_n q^n$ \cite{Dvir08}. Here and throughout the paper, $A\gtrsim_\chi B$ denotes the inequality $A\geq C B$ for some postitive constant $C$ that depends on $\chi.$ The constant for the Kakeya problem over finite fields has been improved since Dvir's work, see e.g. \cite{DKSS13,BukhChao21}. 

It is thus natural to wonder if there is a finite field analogue of a Furstenberg set. Let $s\in (0,1]$ and let $t\in (0,2(n-1)]$. In the finite field setting, we say $F\subset \mathbb{F}_q^n$ is an $(s,t)$-Furstenberg set if there exists a set of (affine) lines $\mathcal L$ such that $|\mathcal L|\gtrsim q^t$ and $|F\cap \ell| \gtrsim q^s$. Based on the work of Orponen--Shmerkin and Ren--Wang, it is then natural to conjecture that if $q$ is prime, an $(s,t)$-Furstenberg set $F\subset \mathbb{F}_q^2$ satisfies
\[
|F| \gtrsim q^{\min\{s+t,\frac{3s+t}{2}, s+ 1\}-\epsilon}
\]
for all $\epsilon>0.$ Firstly, note that this conjecture is false if $q$ is not prime due (more or less) to the existence of subfields \cite[Remark 2.1]{Wolff99}. Secondly, note that while the $(s,t)$-Furstenberg set problem is resolved over $\R^2$, the above conjecture is not known in $\mathbb{F}_q^2$. We discuss progress on the $(s,t)$-Furstenberg set problem over finite fields in the following subsections.

\subsection{The \texorpdfstring{$(n,k)$}{(n,k)}-Besicovitch Conjecture and Furstenberg Sets Associated to Flats}

It is then natural to ask what happens when we replace lines with higher dimensional flats.

A higher dimensional analogue of Kakeya sets are known as $(n,k)$-Besicovitch sets, where one replaces line segments in all directions with $k$-dimensional unit discs in all ``directions.'' More precisely, an $(n,k)$\textit{-Besicovitch set} is a compact set $K\subset\R^n$ such that for each $k$-dimensional subspace $P$ in $\R^n$, there exists a translation vector $a_P$ such that $(P\cap B^n(0,1/2))+a_P$ is a subset of $K$. Similar to the Kakeya conjecture, one may ask how small an $(n,k)$-Besicovitch set can be. 

As discussed previously, work of Besicovitch shows that there are measure zero $(n,1)$-Besicovitch sets in all dimensions, however it is conjectured that this doesn't hold for higher dimensional flats. In particular, the $(n,k)$\textit{-Besicovitch conjecture} suggests that there are \textit{no} measure zero $(n,k)$-Besicovitch sets for $k\geq 2$. In the Euclidean setting, work of Bourgain \cite{Bourgain1991BesicovitchTM} and Oberlin \cite{Oberlin2005BoundsFK} shows that the $(n,k)$-Besicovitch conjecture holds for $k >k_0$ where $k_0 := k_0(n)$ is the smallest positive integer satisfying $n\leq \frac{7}{3} 2^{k_0-2} + k_{0}$. 
We will make use of Bourgain and Oberlin's work to obtain our main result. It is worth noting that analogous statements of the $(n,k)$-Besicovitch set conjecture have been explored in non-Euclidean settings, including the work of the second author (see e.g. \cite{EOT10, Dummit2013-oj, Lewko2019,  dhar23nkbesicovitch}). Note that the $(n,k)$-Besicovitch set conjecture is clearly strictly stronger than the statement: every $(n,k)$-Besicovitch set $K\subset \R^n$ satisfies $\dim K = n$ (though both statements are currently unknown over $\R^n$ for small $k$).

Analogously, the Furstenberg set problem associated to $k$-flats asks how small the Hausdorff dimension of an $(s,t;k)$-Furstenberg set can be. 

\begin{define}[$(s,t;k)$-Furstenberg sets]
Let $k\in [1,n-1]$ be an integer, $s\in (0,k]$, and $t\in (0,(k+1)(n-k)]$. We say a Borel set $F\subset \R^n$ is an $(s,t;k)$\textit{-Furstenberg set} if there exists a non-empty set of (affine) $k$-dimensional flats $\mathcal P$ such that $\dim \mathcal P \geq t$ and $\dim (F\cap P)\geq s$ for all $P\in \mathcal P.$
\end{define}

\noindent In this way, $(s,t)$-Furstenberg sets are $(s,t;1)$-Furstenberg sets, and $(n,k)$-Besicovitch sets are $(k,k(n-k);k)$-Furstenberg sets. However, again note that the family of $k$-flats given by an arbitrary $(s,t;k)$-Furstenberg set \textit{need not} be spread out in different directions.

Finding sharp lower bounds for the Furstenberg set problem associated to $k$-flats is of great interest. Some of the first results of this kind were obtained by Oberlin \cite{Oberlin2007} and Falconer and Mattila \cite{Falconer2016}. In the work of Oberlin, instead of assuming that $\dim(F\cap P)\geq s$, he assumes $P\subset F$ for all $P \in \mathcal P$. Under this hypothesis, Oberlin is able to obtain that 
\begin{equation}\label{eqn:OberlinFalconerMattila}
\begin{cases}
    \dim F \geq 2k - k(n-k)+t, & t\leq (k+1)(n-k) - k \\
    F \text{ has positive Lebesgue measure}, &t> (k+1)(n-k)-k
\end{cases}.
\end{equation}
In the work of Falconer and Mattila, instead of assuming $P\subset F$, they assume that a positive-measure subset of $P$ is contained in $F$ and obtain \eqref{eqn:OberlinFalconerMattila}. Since then, H\'era--Keleti--M\'ath\'e \cite{herakeletimathe} weakened the positive-measure hypothesis to a dimensional one. In particular, they show an $(s,t;k)$-Furstenberg set $F\subset \R^n$ satisfies $$\dim F \geq 2s + \min\{t,1\}-k.$$ Subsequent work of H\'era \cite{heraFurstenberg} shows that an $(s,t;k)$-Furstenberg set $F\subset \R^n$ satisfies
\[
\dim F\geq s + \frac{t - (k-\lceil s\rceil)(n-k)}{\lceil s\rceil +1}.
\]
We will also make use of H\'era's work to obtain our main result. 

More work has been done in the setting of Furstenberg sets of hyperplanes (i.e. when $k=n-1$). Work of Dąbrowski--Orponen--Villa \cite{dabrowski2021integrability} shows that an $(s,t;n-1)$-Furstenberg set, $F\subset \R^n$, with $t\in (1,n]$ satisfies 
\[
\dim F\geq 2s + 2-n - \frac{(t-1)(n-1-s)}{n-1}.
\]
An important input into this result involves point-hyperplane duality, which we will utilize and discuss further in Section \ref{sec:appendix}.

As far as the authors know, the only paper discussing (specifically) $(s,t; k)$-Furstenberg sets over finite fields is work of Gan \cite{gan2024furstenberg}. Roughly, Gan shows that assuming one has good lower bounds for $(s,t;1)$-Furstenberg sets, then one has good lower bounds for $(s,t;k)$-Furstenberg sets. That said, much more work has been done regarding Furstenberg sets over finite fields when additional assumptions are places on the set of associated $k$-flats. One such assumption leads us to the main object of study in this paper: \textit{spread Furstenberg sets}.

\subsection{Spread Furstenberg Sets}\label{subsec:1.3}

As we have alluded to, a key difference between the Furstenberg set problem (associated to affine $k$-flats) and the statements of the Kakeya conjecture and the $(n,k)$-Besicovitch conjecture is that, in the conjectures, the $k$-flats lie in different ``directions.'' This difference is precisely the focus of our work on \textit{spread Furstenberg sets} in this paper.

\begin{define}[$(s,t;k)$-spread Furstenberg sets]
    Let $k\in [1,n-1]$ be an integer, $s\in (0,k]$, and $t\in (0,k(n-k)]$. We say a Borel set $F\subset \R^n$ is an $(s,t;k)$-\textit{spread Furstenberg set} if there exists a non-empty set of $k$-dimensional \textit{subspaces} $\mathcal P$ such that $\dim \mathcal P \geq t$, and for all $P \in \mathcal P$, there exists a translation vector $a_P \in \R^n$ such that $\dim (F\cap (P+a_P)) \geq s$.
\end{define}

Notice that $(s,t;k)$-spread Furstenberg sets are also $(s,t;k)$-Furstenberg sets by embedding the translated planes $P + a_P$ into the space of affine $k$-flats. However, while this embedding certainly doesn't decrease the value of $t$, it is possible for the embedding to increase the value of $t$ (in the same way that a graph of a function $f:\R\to \R$ embedded in $\R^2$ can have Hausdorff dimension larger than 1).

With this definition, it is thus natural to ask how small the Hausdorff dimension of a $(s,t;k)$-spread Furstenberg set can be. This variant has been especially explored in the finite field setting, though notably less so in the Euclidean setting. That said, in fact, this is the variant of the Furstenberg set problem (associated to lines) that was initially studied by Wolff \cite{Wolff99}.

In the finite field setting, we say a set $F\subset \mathbb{F}_q^n$ is an $(s,t;k)$-spread Furstenberg set if there exists a set of $k$-dimensional subspaces $\mathcal P$ such that $|\mathcal P|\gtrsim q^t$, and for all $P\in \mathcal P$ there exists a vector $a_P\in \mathbb{F}_q^n$ such that $|F\cap (P+a_P)|\gtrsim q^s$. Here, again, $s\in (0,k]$ and $t\in (0,k(n-k)]$. Note that most work towards this problem involves a full-dimensional set of subspaces, i.e. $t = k(n-k)$.

Let's begin with the history of the spread Furstenberg set problem associated to lines, i.e. $k=1$, over finite fields.
Applying the polynomial method approach that Dvir used to solve the Kakeya problem over finite fields, one can show that an $(s,n-1;1)$-spread Furstenberg set $F\subset \mathbb{F}_q^n$ satisfies 
\[
|F|\gtrsim q^{ns}.
\]
Additionally, a counting argument (counting pairs of points that lie on the same line) yields the lowerbound of $q^{s + \frac{n-1}{2}}$. However, one can obtain better bounds if $q$ is prime. Work of Zhang \cite{Zhang2015} shows that given $q$ is prime, then a $(\frac{1}{2},n-1;1)$-spread Furstenberg set $F\subset \mathbb{F}_q^n$ satisfies
\[
|F|\gtrsim q^{\frac{n}{2} + \Omega(1/n^2)}.
\]
He furthermore shows that there are $(s,n-1;1)$-spread Furstenberg sets $F\subset \mathbb{F}_q^n$ such that 
\[
|F| \lesssim q^{\frac{n+1}{2} s + \frac{n-1}{2}},
\]
and conjectures that this bound is sharp when $q$ is prime. Note that over $\R^n$ similar conjectures have been recently made with additional hypotheses; see for instance \cite[Conjecture 15]{wangwu}.

For larger values of $k$, much has been done (see \cite{EOT10, KLSS2011, EE16}), though for our purposes we focus on the work of the second author, Dvir, and Lund \cite{DDL-2}. In particular, using the bounds from the Kakeya problem (when $k=1$), Dhar, Dvir, and Lund (roughly) show that for $s\geq 1$, $(s,k(n-k);k)$-spread Furstenberg sets $F\subset \mathbb{F}_q^n$ satisfy
\begin{equation}\label{eqn:DDLBound}
|F|\gtrsim q^{n-k+s}.
\end{equation}
Notably, a pigeonholing argument shows that any set of $q^{n-k+s}$ points in $\mathbb{F}_q^n$ is a $(s,k(n-k);k)$-spread Furstenberg set (and thus \eqref{eqn:DDLBound} is sharp). Furthermore, note that $(s,t;k)$-spread Furstenberg sets over finite fields have also been studied for $t< k(n-k)$; see for instance \cite{Dhar_linear}.

In this paper, we follow the methodology of Dhar--Dvir--Lund to obtain an analogous bound over $\R^n$ for sufficiently large $k$-flats.

\subsection{Main Results}

Our main result is the following. 

\begin{thm}\label{thm:ourmainresult}
Let $k_0$ be the smallest positive integer satisfying satisfying $\frac{7}{3} 2^{k_0 -2} + k_0\geq n$. Then, for all $k \geq  k_0 +1$ and $s>k_0$, if $F\subset \R^n$ is a $(s,t;k)$-spread Furstenberg set, then 
\[
\dim F \geq n-k + s - \frac{k(n-k) -t}{\lceil s\rceil - k_0 +1}.
\]
\end{thm}

A few notes about this result before we move on. Firstly, we need take $k$ sufficently large (i.e. larger than $k_0$) so that we can apply the maximal function bounds of Bourgain \cite{Bourgain1991BesicovitchTM} and Oberlin \cite{Oberlin2005BoundsFK} for the $(n,k)$-Besicovitch conjecture. Heuristically, if the Kakeya maximal conjecture was known, one should be able to take $k_0 = 1.$ Secondly, just as in the discrete setting, when $t = k(n-k)$, this bound is sharp by the Marstrand slicing theorem (another key ingredient in our approach). This is discussed further in Section \ref{sec:preliminaries}. Lastly, a key ingredient in the proof of Theorem \ref{thm:ourmainresult} is the bound for (not necessarily spread) $(s,t;k)$-Furstenberg sets obtained by H\'era \cite{heraFurstenberg}. 

Via our methodology, we are able to apply maximal function bounds of Cordoba \cite{cordobakakeya} to obtain the following for spread-Furstenberg sets associated to \textit{hyperplanes}.

\begin{thm}\label{thm:mainhyperplane}
    For $n\geq 3$, $s\in (1,n-1]$, and $t\in (0,n-1]$, any $(s,t;n-1)$-spread Furstenberg set $F\subset \mathbb{R}^n$ satisfies
    \[
    \dim F \geq 1 + s - \frac{n-1-t}{\lceil s\rceil}.
    \]
\end{thm}

Surprisingly, Theorem \ref{thm:mainhyperplane} gives new bounds for $(s,t;n-1)$-Furstenberg sets by the following Proposition.

\begin{prop}\label{prop:projectivetrans}
    Let $F \subset \R^n$ be an $(s,t;n-1)$-Furstenberg set. Then, there exists a projective transformation $\phi$ such that $\phi(F)$ is an $(s,\min\{t,n-1\};n-1)$-spread Furstenberg set.
\end{prop}

As an immediately corollary, we obtain the following.

\begin{cor}\label{cor:hyperplaneintro}
    Let $n\geq 3$, $1<s$, and let $F\subset \R^n$ be an $(s,t;n-1)$-Furstenberg set. Then,
    \[
    \dim F\geq 1 + s - \frac{n-1 - \min\{t,n-1\}}{\lceil s\rceil}.
    \]
\end{cor}

\begin{rem}
    The above bound is stronger than H\'era's and Dąbrowski--Orponen--Villa's if $t\geq n-1$. If $t< n-1$, our bound is stronger than H\'era's if $n-1 \leq \lceil s\rceil + t$. On the other hand, for $1<t<n-1$, Dąbrowski--Orponen--Villa's bound converges to $n$ as $s\nearrow n-1$ while our estimate does not. In these two ways, Corollary \ref{cor:hyperplaneintro} is neither strictly stronger nor weaker than previously known results.
\end{rem}

\begin{proof}[Proof of Corollary \ref{cor:hyperplaneintro}]
This follows immediately by Theorem \ref{thm:ourmainresult} and Proposition \ref{prop:projectivetrans}, noting that projective transformations preserve Hausdorff dimension (i.e. $\dim F = \dim \phi(F)$).
\end{proof}

The proof of Proposition \ref{prop:projectivetrans} is contained in Section \ref{sec:appendix}. The proof follows via an application of point-hyperplane duality and Marstrand's projection theorem. Note that similar statements for arbitrary $k<n-1$ fails. In particular, given an $(s,t;k)$-Furstenberg set, it is possible that there exists a higher dimensional flat in which the associated set of $k$-flats lie, which implies there isn't a projective transformation $\phi$ satisfying the above condition. That said, it would be interesting to know if there exists a non-concentration condition on $k$-flats that overcomes this barrier. More precisely, answers to the following (vague) question would be of interest:

\begin{question}
Let $F\subset \R^n$ be a $(s,t;k)$-Furstenberg set with an associated set of $k$-flats $\mathcal P$. Is there a non-concentration assumption one can place on $\mathcal P$ so that there exists a projective transformation $\phi$ such that $\phi(F)$ is a $(s, \min\{t,k(n-k)\};k)$-spread Furstenberg set? Or, more generally, such that $\phi(F)$ is a $(s,\min\{t, k(m-k)\};k)$-spread Furstenberg set where $k+1 \leq m \leq n$?
\end{question}

\noindent We do not explore this further here.

\subsection{Paper Outline}

In Section \ref{sec:preliminaries}, we discuss how one defines the metrics (and relatedly the Hausdorff dimension and measures on) the set of $k$-dimensional subspaces in $\R^n$ (the Grassmannian) and the set of $k$-dimensional affine flats in $\R^n$ (the affine Grassmannian). We also state a few results from geometric measure theory that we will black box for this paper (in particular, Marstrand's slicing theorem, H\'era's bound for $(s,t;k)$-Furstenberg sets, and the Kakeya maximal bounds of Bourgain and Oberlin). In Section \ref{sec-overview}, we discuss the methodology for obtaining Theorem \ref{thm:ourmainresult} which is largely motivated by \cite{DDL-2}. As a warm up, we also prove a subcase of Theorem \ref{thm:ourmainresult} when $t = k(n-k)$. Then, in Section \ref{sec:mainresultproof}, we prove Theorem \ref{thm:ourmainresult} for the entire range of $t$. Lastly, in Section \ref{sec:appendix}, we prove that every $(s,t;n-1)$-Furstenberg set in $\R^n$ is an $(s,\min\{t,n-1\};n-1)$-spread Furstenberg set using point-hyperplane duality and Marstrand's projection theorem.

\vspace{1em}
\begin{sloppypar}
\noindent {\bf Acknowledgements.}
The authors would like to thank Larry Guth, Pablo Shmerkin, Josh Zahl, Zeev Dvir, and Simone Maletto for useful discussions. We would also like to thank Damian Dąbrowski for bringing our attention to the hyperplanar Furstenberg set bound of himself, Orponen and Villa in comparison with our Corollary \ref{cor:hyperplaneintro}.
\end{sloppypar}

\section{Preliminaries} \label{sec:preliminaries}

We spend this section carefully defining the notion of a metric on the set of $k$-dimensional subspaces and the set of affine $k$-flats in $\R^n.$ With this notion of a metric, we will also briefly discuss background regarding Hausdorff dimension. This discussion will include both classical results in this area (that will allow us to make standard reductions to our proof), as well as the important results of H\'era, Bourgain, and Oberlin that we noted in the introduction.

Let $1\leq k \leq n-1$ be an integer. Then, we denote the set of $k$-dimensional \textit{subspaces} in $\R^n$, i.e. the \textit{Grassmannian}, as $\mathcal G(n,k)$. The metric on $\mathcal G(n,k)$ can be seen via the operator norm between orthogonal projections. In particular, given $U\in \mathcal G(n,k)$, let $\pi_U:\R^n \to U$ denote the orthogonal projection onto $U.$ Then, given $U,U'\in \mathcal G(n,k)$, we can define the distance between $U,U'$ as $d_{\mathcal G}(U,U') = \lVert \pi_U - \pi_{U'}\rVert_{\textrm{op}}$. One can check that this is in fact a metric on $\mathcal G(n,k)$. Furthermore, $\mathcal G(n,k)$ is in fact a compact smooth manifold of dimension $k(n-k)$. One quick way to see this is by noting $\mathcal G(n,k) \cong O(n) /(O(k)\times O(n-k))$. This in fact shows that $\mathcal G(n,k)$ is compact, and bestows on it a natural Haar measure which we denote $\gamma_{n,k}$ (see e.g. \cite[Chapter 3]{mattila_geometryofsets} for more). 

We similarly define the set of \textit{affine} $k$\textit{-flats} in $\R^n$, denoted $\mathcal A(n,k)$ for the \textit{affine Grassmannian}. An affine $k$-flat is a translation of a $k$-dimensional subspace $U \in \mathcal G(n,k)$. In particular, if $W \in \mathcal A(n,k)$ is a translation of the subspace $U$, we can uniquely identify $W$ by $U$ and the unique $a_W \in U^\perp$ such that $W = U + a_W$. With this identification, it is thus natural for the metric on $\mathcal A(n,k)$ to be defined as follows: given $W,W'\in \mathcal A(n,k)$ with $W = U + a_W$ and $W' = U' + a_{W'}$, we define 
\[
d_{\mathcal A}(W,W')= \lVert \pi_U - \pi_{U'} \rVert_{\textrm{op}} + |a_W - a_{W'}|.
\]
Furthermore, it is natural to define a product measure on $\mathcal A(n,k)$ as being $\lambda_{n,k} = \mathcal L^{n-k} \times \gamma_{n,k}$, where here $\mathcal L^{n-k}$ is the Lebesgue measure on $U^\perp \cong \R^{n-k}$. With this product measure, we can define how to integrate a function $f: \mathcal A(n,k) \to \mathbb{C}$. Namely, when defined, we have 
\[
\int_{\mathcal A(n,k)}f(W) \, d \lambda_{n,k}(W) := \int_{\mathcal G(n,k)} \int_{U^\perp} f(U + a) \, da \, d\gamma_{n,k}(U).
\]
One can check that $\dim \mathcal A(n,k) = (k+1) (n-k)$.

We will utilize the following notation regarding metric spaces.

\begin{notation}
Given a metric space $(X,d)$, we denote the ball of radius $\delta$ centered at a point $x\in X$ as $B_X(x,\delta)$ or simply $B(x,\delta)$ if the metric space is clear from context. We more generally denote the $\delta$-neighborhood of a subset $S\subset X$ as $B_X(S,\delta)$, or simply $B(S,\delta)$.
\end{notation}

From here, we note a few simple observations that we will need to obtain our main results. The first observation regards the measure of $\delta$-balls in the Grassmannian and the affine Grassmannian.

\begin{lem}\label{lem-BallVolGrass}
Given $\delta>0$, $U \in \mathcal G(n,k)$, and $W\in \mathcal A(n,k)$, we have
\begin{align*}
    \gamma_{n,k}(B_{\mathcal G(n,k)}(U,\delta)) &\sim \delta^{k(n-k)} \\
    \lambda_{n,k}(B_{\mathcal A(n,k)}(W,\delta)) &\sim \delta^{(k+1)(n-k)}.
\end{align*}
\end{lem}

\begin{proof}
This follows by the definition of the Haar measures, and we omit the proof here.
\end{proof}

The second observation notes how translating subspaces in $\mathcal G(n,k)$ relates to distances in $\mathcal A(n,k)$.

\begin{lem}\label{lem-grassDist}
Given $U,V\in \mathcal G(n,k)$ and $a\in U^\perp$ we have,
$$d_{\mathcal A}(U+a,V+a)\le (|a|+1)d_{\mathcal G}(U,V).$$
\end{lem}

\begin{proof}
This follows by noting that $V + a = V + (a-\pi_V(a))$ and $a-\pi_V(a) \in V^\perp$. Hence, for all $a \in U^\perp$
\[
|\pi_V(a)| = |a-(a-\pi_V(a))| = |(\pi_U - \pi_V)(a)| \leq \lVert \pi_U - \pi_V\rVert_{\mathrm{op}} \,|a|.
\]
Therefore, 
\[
d_{\mathcal A}(U + a, V+ a) = d_{\mathcal G}(U,V) + |a-(a-\pi_V(a))| \leq (|a|+1) d_{\mathcal G}(U,V).
\]
\end{proof}

Given $U,V\in \mathcal G(n,k)$, we let $R_{U,V}\in O(n)$ be the operator such that $V=R_{U,V}U$ and $\|I-R_{U,V}\|_{op}$ is minimized. This also defines a metric over the Grassmanian which is equivalent to the metric we defined earlier. In particular, we will need the following.

\begin{lem}\label{lem-rotatUV}
Given $U,V\in \mathcal G(n,k)$ and $b\in U$ we have the following,
$$|(I-R_{U,V})b|\lesssim_{n} |b|d_{\mathcal G}(U,V).$$
\end{lem}
For reference, see Sections 2 and 5 in \cite{Qiu2005UnitarilyIM}. Using the above, we can now compare the distances of flats within $U$ and $V$ which are related by $R_{U,V}$. For any $U\in \mathcal A (n,k)$ and $k'< k$, we let $\mathcal A (U,k'),$ be the set of affine $k'$-flats contained in $U$. We similarly define $\mathcal G(U,k')$ for $U \in \mathcal G(n,k)$.

\begin{lem}\label{lem-containDist}
Given $U,V\in \mathcal G(n,k)$, let $r>0$, $a\in U^\perp$, and $k'<k$. For any $X\in \mathcal A(U,k')$ such that $X\cap B(0,r) \ne \emptyset$ we have,
$$d_{\mathcal A}(R_{U,V}X+a,X+a)\lesssim_n (r+|a|+1)d_{\mathcal G}(U,V).$$
\end{lem}

\begin{proof}
Let $R$ denote $R_{U,V}$ and let $X=X_0+b$ where $b\in X_0^\perp \cap U$ and $X_0\in \mathcal G(U,k')$. Notice that by assumption, $b\in B(0,r)$ and we have $$RX+a=RX_0+Rb+a-\pi_{RX_0}(a)$$ with $a-\pi_{RX_0}(a)\in (RX_0)^\perp$ and $Rb \in (RX_0)^\perp \cap U$. 
We have 
\begin{equation}\label{eq-cD1}
d_{\mathcal A}(X+a,RX+a)=\|\pi_{X_0}-\pi_{RX_0}\|+|b+a-Rb-(a-\pi_{RX_0}(a))|.
\end{equation}

As $X_0$ and $RX_0$ are contained in $V$ and $U$ respectively, we have 
\begin{equation}\label{eq-cD2}
\|\pi_{X_0}-\pi_{RX_0}\|\le \|\pi_V-\pi_U\|=d_{\mathcal G}(V,U).
\end{equation}
We also have, 
\begin{equation}\label{eq-cD3}
|b+a-Rb-(a-\pi_{RX_0}(a))|\le |(I-R)b|+|\pi_{RX_0}(a)|\lesssim_n (|b|+|a|)d_{\mathcal G}(X_0,RX_0),
\end{equation}
where we used Lemma~\ref{lem-rotatUV} in the last inequality.
Putting \eqref{eq-cD1}, \eqref{eq-cD2}, and \eqref{eq-cD3} together we are done.
\end{proof}

\subsection{Hausdorff Dimension}

We now recall the definition of Hausdorff measure and Hausdorff dimension on arbitrary metric spaces. Let $(X,d)$ be a metric space and let $A\subset X$. Note that some texts require that $X$ be separable, though this will not raise a problem in our setting as all of the metric spaces we will consider here are separable. Then, we denote the \textit{$s$-dimensional Hausdorff measure of $A$ at scale} $\delta$ as 
\[
H^s_\delta(A) := \inf \left\{\sum_{i=1}^\infty (\mathrm{diam} \,U_i)^s : \bigcup_{i=1}^\infty U_i \supset A, \mathrm{diam}\,U_i < \delta\right\}.
\]
Then, the $s$\textit{-dimensional Hausdorff measure} of $A$ is given by $H^s(A) = \lim_{\delta \to 0} H^s_\delta(A).$ Notice that the $n$-dimensional Hausdorff measure on $\R^n$ is (up to some multiplicative constant) the Lebesgue measure. Lastly, we define and denote the \textit{Hausdorff dimension} of $A$ as 
\[
\dim A = \inf \{s>0 : H^s(A) = 0\} = \sup\{s>0 : H^s(A) = \infty\}.
\]
Hence, every set of positive Lebesgue measure in $\R^n$ has Hausdorff dimension $n$. There are numerous other similar notions of dimension (e.g. packing dimension and box dimension), though we don't go into this further here.

There are a few key facts specifically regarding Hausdorff dimension that we will need to prove our main result. The first two are standard results in this area.

\begin{lem}\label{lem:posHausMea}
Let $(X,d)$ be a metric space. Given a Borel set $A\subset X$ such that $H^s(A)>0$, there exists a compact set $A'\subset A$ such that $0< H^s(A')<\infty$.
\end{lem}

\begin{proof}
This follows from the Borel regularity of Hausdorff measure, see \cite[Chapter 4]{mattila_geometryofsets}.
\end{proof}

Furthermore, we need a slicing theorem, which heuristically states given $A\subset \R^n$ Borel with $\dim A >s$, then almost every ``slice'' of $A$ by an $m$-flat has dimension at least $s-m$.

\begin{thm}[\cite{mattilaBook}, Theorem 6.8]\label{thm:mattilaSlice}
Let $k\le s \le n$ and let $A \subseteq \R^n$ be a Borel set with $\dim A >s$.
Then there is a Borel set $E \subset \mathcal G(n, k)$ such that
$$\dim E \le k(n-k) + k - s$$
and
$\cH^k(\{a\in U : \dim A \cap (U^\perp + a) > s - k\}) > 0$ for all $U \in \mathcal G(n, k) \setminus E$.
\end{thm}

\noindent See \cite[Chapter 6]{mattilaBook} for a proof and detailed history of the above slicing theorem. We note here that this theorem implies that our main result (Theorem \ref{thm:ourmainresult}) is sharp when $t = k(n-k)$. In particular, for all $A\subset \R^n$ Borel with $\dim A \geq n-k+s$, it follows that $A$ is an $(s,k(n-k);k)$-spread Furstenberg set. To see this, firstly assume with loss of generality that $\dim A = n-k+s$, and let $s' < s$ be arbitrary. Then, by Theorem \ref{thm:mattilaSlice}, it follows that for almost every $U\in \mathcal G(n,k)$,
\[
H^{n-k}(\{a\in U^\perp : \dim A\cap (U+a) >s'\}) >0.
\]
Hence, for almost every $U\in \mathcal G(n,k)$, there exists a translation vector $a_U \in U^\perp$ such that $\dim A\cap (U + a_U)>s'$. This shows that $A$ is an $(s',k(n-k);k)$-spread Furstenberg set. Since $s'<s$ was arbitrary, it follows that $A$ is an $(s,k(n-k);k)$-spread Furstenberg set. Thus our main theorem, which implies that 
\[
\dim A \geq n-k+s,
\]
is sharp. 

Lastly, a key theorem we need to obtain our main result is a dimensional bound of H\'era for $(s,t;k)$-Furstenberg sets.

\begin{thm}[\cite{heraFurstenberg}, Theorem 1.4]\label{thm-heraFurst}
Let $0 < s \le k$, and $0 \le t \le (k + 1)(n - k)$ be any real numbers. Suppose that $F\subset \R^n$ is an $(s,t;k)$-Furstenberg set. Then

$$ \dim F \ge s + \frac{t-(k -\lceil s \rceil)(n - k)}{
\lceil s \rceil + 1}.$$
\end{thm}

\noindent We black box the above result for the purposes of this paper.

\subsection{Maximal Function Bounds}

The last result we need to obtain our main result is a maximal function bound obtained by Bourgain \cite{Bourgain1991BesicovitchTM} and Oberlin \cite{Oberlin2005BoundsFK}. To motivate this result, we begin with a discussion of the Kakeya maximal conjecture.

The Kakeya maximal conjecture arose in the study of the Kakeya conjecture. In particular, let $K \subset \R^n$ be a Kakeya set---a set with a line segment in each direction. To study the Hausdorff dimension of $K$, we can heuristically think of each line segment as a $\delta$-neighborhood of the line segment. This $\delta$-neighborhood is often referred to as a $\delta$\textit{-tube.}

\begin{notation}[$\delta$-tubes]
    Given $U \in \mathcal G(n,1)$ and $a\in U^\perp$, we let $T(U+a,\delta)$ denote the $\delta$\textit{-tube} (centered at $a$ and in the direction of $U$). Precisely, $T(U+a,\delta)$ is the $\delta$-neighborhood of $(U+a) \cap B(a,1/2)$.
\end{notation}

Then, we can define what is referred to as the \textit{Kakeya maximal function}. Given a locally integrable function $f: \R^n \to \mathbb{C}$, we define $\mathcal M_\delta^1f: \mathcal G(n,1) \to \R$ via the equation 
\begin{equation}\label{eqn:tubemax}
\mathcal M_\delta^1f(U) := \sup_{a\in U^\perp} |T(U+a,\delta)|^{-1} \int_{x\in T(U+a,\delta)} |f(x)|\,dx.
\end{equation}
Then, the Kakeya maximal conjecture states the following.

\begin{conjecture}[Kakeya maximal conjecture]
    For all $\epsilon>0$, there exists a constant $C_\epsilon>0$ such that for any function $f: \R^n \to \mathbb{C}$ and all $\delta>0$,
    \[
    \lVert \mathcal M_\delta^1(f)\rVert_{L^n(\mathcal G(n,1))} \leq C_\epsilon \delta^{-\epsilon} \lVert f\rVert_{L^n(\R^n)}.
    \]
\end{conjecture}

It is a nontrivial fact that maximal function bounds like the one above can imply Hausdorff dimensional bounds. In fact, the Kakeya maximal conjecture implies the Kakeya conjecture, and the Kakeya maximal conjecture is only known in two dimensions \cite{cordobakakeya}. For more background and new results on maximal functions associated to such geometric objects, see \cite{zahl2023maximalfunctionsassociatedfamilies}.

It is then natural to explore higher dimensional analogues of the Kakeya maximal function. 

\begin{define}[$\delta$-slabs]
    Given $U \in \mathcal G(n,k)$ and $a\in U^\perp$, we let $T(U+a,\delta)$ denote the $\delta$\textit{-slab} (centered at $a$ and in the direction of $U$). Precisely, $T(U+a,\delta)$ is the $\delta$-neighborhood of $(U+a) \cap B(a,1/2)$.
\end{define}

Then, given a locally integrable function $f: \R^n \to \mathbb{C}$, we define the maximal function $\mathcal M_\delta^kf: \mathcal G(n,k)\to \R$ the same way as equation as \eqref{eqn:tubemax} where now $U \in \mathcal G(n,k)$. Using this notation, we can now state the key result of Bourgain and Oberlin we will need for our main result.

\begin{thm}[\cite{Bourgain1991BesicovitchTM,Oberlin2005BoundsFK}] \label{thm:realKakeyalog}
Let $k_0$ be the smallest positive integer such that $n \leq \frac{7}{3}2^{k_0 - 2} + k_0$. Then, fix $k \geq k_0 +1$ and let $\mathcal M_\delta^k$ be the maximal function for $k$-flats. Then, for all $\epsilon>0$, 
\[
\lVert \mathcal M_\delta^{k} f\rVert_{L^{(n-1)/2}(\mathcal G(n,k))} \lesssim_n \delta^{-\epsilon} \lVert f\rVert_{L^{(n-1)/2}(\R^n)}.
\]
\end{thm}

It may be helpful to note for intuition that the above theorem holds e.g. for $k\geq \lceil \log_2 n\rceil$. This will be a key result we need for our main theorem. Before moving on, we give some quick historical remarks. Firstly, the above bound was obtained by Bourgain assuming $n \leq 2^{k_0-2} + k_0$ using the Kakeya maximal function bounds (associated to lines); the larger range of $k_0$ was obtained by Oberlin by using better bounds for the Kakeya maximal function by Katz and Tao. Secondly, it may be worth noting that the above result as stated follows as (as a corollary) of bounds for the operator
\[
\mathcal N^kf(U) = \sup_{a\in U^\perp} \int_{x\in U} |f(a+x)| \, dx.
\]
Here, again, $f$ is a locally integrable function. In \cite{zahl2023maximalfunctionsassociatedfamilies}, this is referred to as a Bourgain-type maximal function, though we don't discuss this further here. Lastly, note that such maximal functions have also been studied in different geometries (such as finite fields), see e.g. \cite{EOT10}.



\section{Methodology}\label{sec-overview}

Using the tools discussed in the previous section, we now summarize the methodology for studying $(s,t;k)$-spread Furstenberg sets over finite fields due to \cite{DDL-2, Dhar_linear}. We will also take this time to outline where our tools from the Euclidean setting will be applied within this framework. For all that follows, let $k \geq k_0 + 1$ where $k_0$ is the integer defined in Theorem \ref{thm:realKakeyalog}.

\begin{itemize}
    \item \textbf{Concentration step:} Let $P$ be a $k$-flat that intersects with $F$ in dimension at least $s$. We show `most' $(k-k_0)$-flats $U$ contained in $P$ satisfy $\dim (F\cap U) \geq s - k_0$. Here, `most' will mean in a set of positive measure. Over finite fields this follows by a direct secnd moment method argument. In the Euclidean setting, we will use a slicing theorem (Theorem \ref{thm:mattilaSlice}). 
    \item \textbf{Kakeya step:} As we have $k$-flats in many directions intersecting with $F$ in dimension $s$, and each such flat contains a positive measure of $(k-k_0)$-flats intersecting with $F$ in dimension $s-k_0$, we will show a much larger set of $(k-k_0)$-flats contained in $\mathcal A(n,k-k_0)$ intersects with $F$ in dimension at least $s-k_0$. The key idea is that, fixing some $U\in \mathcal G(n,k-k_0)$, any $k$ flat in $\R^n$ containing $U$ contains a $k_0$-flat of shifts of $U$. This allows us to apply the Kakeya maximal function bounds due to Bourgain and Oberlin (Theorem \ref{thm:realKakeyalog}).
    \item \textbf{Affine Furstenberg step:} We now have a large set of \textit{affine} $(k-k_0)$-flats contained in $\mathcal A(n,k-k_0)$ with large intersection with $F$. 
    In particular, we will show that our set is a large $(k-k_0;s-k_0, \alpha)$-Furstenberg set where 
    \[
    \alpha = (k-k_0 + 1) (n-k+k_0) -k(n-k) + t.
    \]
    Over finite fields, another second moment argument is used to show that the resulting Furstenberg set is large in cardinality. In our setting, we will apply the Furstenberg set bound due to H\'{e}ra (Theorem \ref{thm-heraFurst}) to finish the argument.
\end{itemize}

\subsection{Warm up: All Directions}

As a warm up, we prove our main theorem when we have all directions of $k$-flats, i.e. $t = k(n-k).$

\begin{thm}
Let $k_0$ be the smallest positive integer satisfying $n \leq \frac{7}{3}\cdot 2^{k_0-2}+k_0$. Then, fix $k\geq k_0+1$ and let $F \subset \R^n$ be an $(s,t;k)$-spread Furstenberg set whose associated set of $k$-dimensional subspaces is $\mathcal P = \mathcal G(n,k)$. Then $\dim F \ge n-k+s$  
\end{thm}

\begin{proof}
For each $P \in \mathcal G(n,k)$, let $a_P\in \R^n$ be such that $\dim (F \cap (P+a_P)) \geq s$.

Without loss of generality, for all $P \in \mathcal G(n,k)$ we can assume not only that $\dim(F\cap (P + a_P))\ge s$, but also $0<H^s(F\cap (P+ a_P))<\infty$. This reduction follows by looking at $s'=s-\epsilon$ which means $H^{s'}(F\cap (P + a_P))=\infty$ and using Proposition~\ref{lem:posHausMea}.

Let $\cS$ be the set of flats in $\mathcal A(n,k-k_0)$ such that their intersection with $F$ has dimension at least $s-k_0$. We claim that $\cS$ has positive measure. The result then follows from applying the H\'era's Furstenberg set bound (Theorem \ref{thm-heraFurst}).

For every $V \in \mathcal G(n,k-k_0)$, the set of flats parallel to $V$ in $\mathcal A(n,k-k_0)$ is isomorphic to the space $\R^{n-k+k_0}\cong \R^n/V$. For any flat $U\in G(\R^n/V,k_0)$ we can find a shift $a_{V,U}$ such that $a_{V,U}+(\Span\{U,V\})$ intersects with $F$ in positive $H^s$ measure. By the slicing theorem (Theorem \ref{thm:mattilaSlice}), a positive measure of $(k-k_0)$-flats in $a_{V,U}+(\Span\{U,V\})$ intersect with $F$ in dimension at least $s-k_0$.

This gives a $k_0$-flat in $\R^n/V$ parallel to $U$ which contains a positive measure of points contained in $\cS$ restricted to $\R^n/V$. As this is true for all $U\in G(\R^n/V,k_0)$, $\cS$ restricted to $\R^n/V$ has positive measure by Bourgain and Oberlin's Kakeya maximal bounds (Theorem \ref{thm:realKakeyalog}). As this is true for all $U\in \mathcal  G(\R^n/V,k_0)$, we are done.
\end{proof}

\section{Proof of Main Theorem}\label{sec:mainresultproof}

The proof of the main theorem is essentially the same as our warm up proof. The extra details are to work with Hausdorff dimension via $\delta$-discretization in standard ways.

\begin{thm}\label{thm-main}
Let $k_0$ be the smallest positive integer such that $n \leq \frac{7}{3} 2^{k_0 - 2} + k_0$. Then, fix $k \geq k_0 + 1$, $s\in (k_0,k]$ and $t\in (0,k(n-k)]$. Then, any $(s,t;k)$-spread Furstenberg set $F\subset \R^n$ satisfies
\[
\dim F \geq n-k + s - \frac{k(n-k) - t}{\lceil s\rceil -k_0 +1}.
\]
\end{thm}

\begin{proof}
Let $F$ be a $(s,t;k)$-spread Furstenberg set satisfying the above hypotheses. Let $\mathcal P \subset \mathcal G(n,k)$ denote the associated set of $k$-dimensionsl subspaces with associated translation vectors $\{a_P\}_{P\in \mathcal P}$ such that $\dim (F\cap (P+a_P)) \geq s$.

Without loss of generality we can assume $F$ is bounded. To see this, let $\epsilon > 0$ and consider $F_i = F\cap B(0,i), i\in \mathbb N$. We have $\bigcup_i F_i = F$. For each $F_i$, let $\mathcal P_i$ be the set of $k$-dimensional subspaces such that $\dim (F_i \cap (P+a_P))\geq s-\epsilon$. We have $\bigcup_i \mathcal P_i \supseteq \mathcal P$. Furthermore, $P_i$ may not be Borel but $F_i$ clearly is. By the countable sub-additivity of the Hausdorff measure we must have an $i\in \mathbb{N}$, such that $F_i$ is a $(s-\epsilon,t-\epsilon;k)$-spread Furstenberg set. From now on we assume that $F$ is within $B(0,1/8)$. Furthermore, as in the warm-up, we can assume not only $\dim(F \cap (P + a_P))\ge s$ but also $0< H^s(F\cap (P + a_P) )<\infty$. This follows by letting $s'=s-\epsilon$ (so that $H^{s'}(F\cap (P+ a_P))=\infty$) and applying Proposition~\ref{lem:posHausMea}.

Let $\cS\subseteq \mathcal A(n,k-k_0)$ be the set of $(k-k_0)$-flats which intersect with $F$ in dimension at least $s-k_0$. We claim that $\cS$ has dimension at least $$\alpha = (k-k_0+1)(n-k+k_0)-k(n-k) + t.$$ This and H\'era's Furstenberg set bound (Theorem~\ref{thm-heraFurst}) will complete the proof.

For each $P\in \cP$, we know $P+a_P$ intersects with $F$ in positive $H^s$ measure. Hence, for all $P \in \cP$, Theorem~\ref{thm:mattilaSlice} implies that a positive measure $t_P>0$ of $(k-k_0)$-flats in $P + a_P$ intersect with $F$ with dimension at least $s-k_0$. Given such a flat $P\in \mathcal G(n,k)$, we let $\cH_P\subset \mathcal G(n,k-k_0)$ be the set of $k-k_0$ flats contained in $P$ which, after a shift by $a_P$, intersect with $F$ with dimension at least $s-k_0$. By construction, $|\cH_P|\ge t_P$. We let $$a_P+\cH_P=\{a_P+V|V\in \cH_P\}$$
and
$$\cP_\ell=\{P\in \cP| t_P\ge 1/\ell\}.$$
As $\bigcup_{\ell=1}^\infty \cP_\ell = \cP$, we have for every $\eta >0$ there exists an $\ell_\eta$ such that $\cP_{\ell_\eta}$ has dimension at least $\dim \cP -\eta=t-\eta$. We fix an arbitrary $\eta>0$ and let $\cP'=\cP_{\ell_\eta}$, $\ell'=\ell_\eta$, and $t'=t-\eta$.

Consider a covering of $\cS$ by balls $B_1,\hdots,B_j,\hdots$ of radius at most $r_j<1$. Let $J_a = \{j| 2^{-a}< r_j\le  2^{-a+1}\}$. Define $\cP'_a$ as the set of $P\in \mathcal G(n,k)$ such that 
$$\left|(a_P+\cH_{P})\cap \bigcup\limits_{j\in J_a} B_j\right| \ge \frac{t_P}{100 a^2}.$$
As $\sum_a 1/a^2 < 100$ and $$\sum_a \left|(a_P+\cH_P)\cap \bigcup_{j\in J_a} B_j\right| \ge |\cH_P|\ge t_P>0$$ we have that $\bigcup_a \cP'_a=\cP'$.

Fix an $a$ and let $f_{2^{-a}}$ be the indicator function of the union of the balls $\bigcup_{j\in J_a} B_j \subset \mathcal A(n,k-k_0)$. For every subspace $U\in \mathcal G(n,k-k_0)$, the set of flats parallel to $U$ in $\mathcal A(n,k-k_0)$ is isomorphic to $\R^{n-k+k_0}\cong U^\perp$. We also see that for a subspace $U' \in \mathcal G(n,k)$ parallel to $U$ and $a\in \R^n$, $a+U'/U$ corresponds to a $k_0$-flat in $\R^n / U$. Thus, we have
\begin{equation}\label{eq-volBreak}
\int\limits_{W\in \mathcal A(n,k-k_0)} |f_{2^{-a}}(W)| \, d\lambda_{n,k-k_0}(W) = \int\limits_{U\in \mathcal G(n,k-k_0)} \int\limits_{x \in \R^n/U} f_{2^{-a}}(x+U) \, d x d \gamma_{n,k-k_0}(U).
\end{equation}
For ease of notation, we stop using the notation $\lambda_{n,k}$ and $\gamma_{n,k}$ as it will be clear which measure to integrate against from context.
Given $U \in \mathcal G(n,k-k_0)$ we define the maximal function $\cM_{\delta}^{k_0,U},$ which takes in locally integrable functions over $\mathcal A(n,k-k_0)$ and outputs a function over the set of $V\in \mathcal G(n,k)$ containing $U$. In particular, let
$$\cM_{\delta}^{k_0,U} f(V) = \sup_{a \in U^\perp} \int_{W\in T(V/U+a,\delta)} |T(V/U+a,\delta)|^{-1} f(W) dW, $$
where $U^\perp \cong \R^n/U$ represents the set of $k_0$-flats parallel to $U$. This is just the usual Kakeya maximal function for $k_0$-dim flats, but we are defining a family of these for functions over $\mathcal A(n,k-k_0)$ depending on the direction of $U \in \mathcal G(n,k-k_0)$ we are restricting to. Hence, we can apply Bourgain and Oberlin's Kakeya maximal bounds for $k_0$-flats (Theorem \ref{thm:realKakeyalog}) and Fubini's theorem to obtain (with $p=(n-k_0-1)/2$)
\begin{align}\label{eq-Fubini}
\int\limits_{U\in \mathcal G(n,k-k_0)} \int\limits_{x \in \R^n/U} f_{2^{-a}}(x+U) \, dx dU &\gtrsim_{\eps,p} 2^{ap\eps}\int\limits_{U\in \mathcal G(n,k-k_0)} \int\limits_{V\in \mathcal G(n,k),V\supset U} \left(\cM_{2^{-a}}^{k_0,U} f_{2^{-a}}(V)\right)^p \, dV dU \nonumber\\
&\gtrsim_{\eps,p}  2^{ap\eps} \int\limits_{V\in \mathcal G(n,k)} \int\limits_{U\in \mathcal G(n,k-k_0), U\subset V} \left(\cM_{2^{-a}}^{k_0,U} f_{2^{-a}}(V)\right)^p \,  dU\, dV\nonumber\\
&\gtrsim_{\eps,p}  2^{ap\eps} \int\limits_{V\in \mathcal G(n,k)} \left(\int\limits_{U\in \mathcal G(n,k-k_0), U\subset V} \cM_{2^{-a}}^{k_0,U} f_{2^{-a}}(V) \,\,  dU\right) dV.
\end{align}
The last inequality follows from convexity of $x^p$. Next we prove the following claim which uses the fact that $F$ is bounded.

\begin{claim}
There exists a radius $\delta_a\sim_n 2^{-a}$ such that for any $P\in \cP'_a$ and any $V\in B_{\mathcal G(n,k)}(P,r_a)$ we have,
\begin{align*}
\int\limits_{U\in \mathcal G(n,k-k_0), U\subset V} \cM_{2^{-a}}^{k_0,U} f_{2^{-a}}(V) \,  dU & \gtrsim_{n,k,k_0} \frac{1}{ \ell' a^2}.
\end{align*}
\end{claim}
\begin{proof}
Let $\mathcal G(V,k-k_0)$ denote the set of $(k-k_0)$-flats contained in $V$. By definition,
\begin{align*}
\int\limits_{U\in \mathcal G(V,k-k_0)} \cM_{2^{-a}}^{k_0,U} f_{2^{-a}}(V) \,  dU &\ge \iint\limits_{U\in \mathcal G(V,k-k_0), x\in T(V/U+a_P, 2^{-a})}   |T(V/U+a_P, 2^{-a})|^{-1} f_{2^{-a}}(x)\, dx dU.
\end{align*}
For any $W\in \mathcal A(n,k)$, let $\mathcal A(W,k-k_0)$ denote the set of $k-k_0$ flats in $W$. Any $B_j,j\in J_a$ has radius at least $2^{-a}$. If such a $B_j$ intersects with $\mathcal A(W,k-k_0)$ at $Z$. For any $Z'\in B(Z,2^{-a-4})$, $B_j\cap B(Z',2^{-a})$ contains a ball of radius $2^{-a-1}-2^{-a-5}$ (follows simply by a few applications of the triangle inequality). This means that at least a $\gtrsim_n 1$ fraction of points in $B(Z',2^{-a})$ are in $B_j$. Using this observation we note that
\begin{equation}\label{eq-eqWintEst}
\left| \mathcal A(V+a_P,k-k_0)\cap \bigcup\limits_{j\in J_a} B_j \right| \gtrsim_{n,k,k_0} \frac{1}{\ell' a^2},\end{equation}
implies
$$\int\limits_{U\in \mathcal G(V,k-k_0)} \cM_{2^{-a}}^{k_0,U} f_{2^{-a}}(V) \,  dU \gtrsim_{n,k,k_0} \frac{1}{\ell' a^2}.$$

As $|a_P|\lesssim 1$, we have $d_{\mathcal A}(P+a_P,V+a_P)\lesssim d_{\mathcal G}(P,V)$ by Lemma~\ref{lem-grassDist}.  We can now find a radius $\delta_a\sim_n 2^{-a}$ such that for any $V\in B_{\mathcal G(n,k)}(P,r_a)$, we have $R_{U,V}\cH_P+a_P$ is also contained $\bigcup_{j\in J_a} B_j$ (using Lemma~\ref{lem-containDist} and the fact that we are working within the unit ball). For such a $V$, \eqref{eq-eqWintEst} holds which completes the proof.
\end{proof}


Using this with \eqref{eq-Fubini} we have,
\begin{equation}\label{eq-fin1}
\int\limits_{U\in \mathcal G(n,k-k_0)} \int\limits_{x \in \R^n/U} f_{2^{-a}}(x+U) \, dx dU \gtrsim_{\epsilon, p} \frac{2^{ap\epsilon}}{\ell'^pa^{2p}}|B_{\mathcal G(n,k)}(\cP'_a,\delta_a)|.
\end{equation}
Within $B_{\mathcal G(n,k)}(\cP'_a,r_a)$ we can find a finite covering of $\cP'_a$ into sets with disjoint interiors with diameter $\delta_a'\sim_{n,k,k_0} \delta_a$ (by working with the $\ell_\infty$ norm of the matrices representing the operators and splitting in cubes for instance). Let $\mathcal C(\cP',a)$ be the collection of this covering set. Using Lemma~\ref{lem-BallVolGrass}, the inequality \eqref{eq-fin1} can now be rewritten as,

$$ \sum\limits_{j\in J_a} r_j^{(k-k_0+1)(n-k+k_0)} \gtrsim_{p,\epsilon,n,k,k_0} \frac{2^{ap\epsilon}}{\ell'^pa^{2p}} |\mathcal C(\cP',a)| 2^{-a k(n-k)}. $$
Multiplying the two sides by $2^{a (k(n-k)-t')}$ and rearranging terms gives us,
$$ \sum\limits_{j\in J_a} r_j^{(k-k_0+1)(n-k+k_0)-(k(n-k)-t')} a^{2p}2^{-2ap\epsilon}  \gtrsim_{p,\epsilon,n,k,k_0} \frac{1}{\ell'^p} |\mathcal C(\cP',a)| 2^{-at'}. $$
$a/2^{a\epsilon} \lesssim 1/\epsilon$ (the maximum is achieved for $a\sim 1/\epsilon$). Summing the above over $a$ we have,
$$ \sum\limits_{j=1}^{\infty} r_j^{(k-k_0+1)(n-k+k_0)-(k(n-k)-t')}  \gtrsim_{p,\epsilon,n,k,k_0,\ell'} \sum_{a=1}^{\infty} |\mathcal C(\cP',a)| 2^{-at'}.$$
For $a\geq 1$, the set $\mathcal C(\cP',a)$ contains balls of diameter $\delta'_a\sim_{n,k,k_0} 2^{-a}$ covering $\cP'$. As $\cP'$ has dimension at least $t'=t-\eta$, this shows that $\cS$ has dimension at least $$(k-k_0+1)(n-k+k_0)-(k(n-k)-t+\eta).$$ As $\eps,\eta$ can be made arbitrarily small, this implies that $\cS$ has dimension at least 
$$(k-k_0+1)(n-k+k_0) - k(n-k) + t.$$ Thus, by H\'era's Furstenberg set bound (Theorem~\ref{thm-heraFurst}) we are done.
 \end{proof} 

We see that the above argument uses the maximal function bounds for $k$-flats in a blackbox fashion. If they were known for a smaller $k_0$ then our results will hold for all $k\ge k_0+1$ (we want maximal function bounds for $k_0$ flats in $\R^{n-k+k_0}$). To be precise, we have the following statement (we omit the proof as it is identical to the proof of Theorem~\ref{thm-main}).

\begin{thm}
Let $n_0,k_0$ be integers and $\mathcal M_\delta^{k_0}$ be the maximal function for $k_0$-flats over $\R^{n_0}$. If for all $\epsilon>0$ and some $p>0$ the following holds, 
\begin{equation}\label{eqn:generalmaxstatement}
\lVert \mathcal M_\delta^{k_0} f\rVert_{L^{p}(\mathcal G(n_0,k_0))} \lesssim_{n_0} \delta^{-\epsilon} \lVert f\rVert_{L^{p}(\R^{n_0})},
\end{equation}
then for $k = k_0 + 1$, $n=1+n_0$, $s\in (k_0,k]$, and $t\in (0,k(n-k)]$, any $(s,t;k)$-spread Furstenberg set $F\subset \R^n$ satisfies
\[
\dim F \geq n-k + s - \frac{k(n-k) - t}{\lceil s\rceil -k_0 +1}.
\]
\end{thm}

Outside of Theorem~\ref{thm-main} we can instantiate the above with the maximal Kakeya bounds for tubes in planes due to Cordoba~\cite{cordobakakeya} to get dimension lower bounds for spread Furstenberg sets associated to hyperplanes. In particular, Cordoba proved that \eqref{eqn:generalmaxstatement} holds for $k_0 = 1$ and $n_0 = 2$. Thus, we obtain Corollary \ref{cor:hyperplaneintro} which we restate here.

\begin{cor}
For $n\ge 3$, $s\in (1,n-1]$,  and $t\in (0,n-1]$, any $(s,t;n-1)$-spread Furstenberg set $F\subset \R^n$ satisfies
\[
\dim F \geq 1 + s - \frac{n-1-t}{\lceil s\rceil}.
\]
\end{cor}

In fact, for all values of $s \in (1,n-1]$, there is a value of $t$ such that there exists an $(s,t;n-1)$-spread Furstenberg set such that the above inequality is sharp. To see this, let $X\subset \R^n$ be an $s$-dimensional set contained in a $\lceil s\rceil$-dimensional subspace. I.e., there exists an $H \in \mathcal G(n,\lceil s\rceil)$ such that $X\subset H$. Then, take $\mathcal P$ to be the set of hyperplanes that contain $X$. Notice that $t:= \dim \mathcal P = n-1 - \lceil s\rceil$ as $\mathcal P$ is homeomorphic to $\mathcal G(n-\lceil s\rceil, n-1-\lceil s\rceil).$ Hence, $X$ is an $s$-dimensional $(s,t;n-1)$-spread Furstenberg set, which agrees with the above bound.

\section{Hyperplanes: Spread Furstenberg Sets to Furstenberg Sets} \label{sec:appendix}

In this section, we prove Proposition \ref{prop:projectivetrans}, which we restate here for convenience.

\begin{prop}\label{prop:appendixstatement}
    Let $F \subset \R^n$ be an $(s,t;n-1)$-Furstenberg set. Then, there exists a projective transformation $\phi$ such that $\phi(F)$ is an $(s,\min\{t,n-1\};n-1)$-spread Furstenberg set.
\end{prop}

\noindent Our proof utilizes point-hyperplane duality and Marstrand's projection theorem.

Point-hyperplane duality has proven to be a helpful tool in this area of harmonic analysis and geometric measure theory. Full details of this duality are contained in \cite[Sections 6.1-6.2]{dabrowski2021integrability}, but the core idea can be seen in the plane. Over $\R^2$, one can identify points $(m,b) \in \R^2$ with the lines $y = mx + b$, and vice versa. In particular, let \[\mathbf{D}:\R^2 \ni (m,b) \mapsto \{y = mx + b: x\in \R\}\in \mathcal A(2,1)\]
and 
\[
\mathbf{D}^\ast : \mathrm{im}\, \mathbf{D}\ni \{y = mx + b\}\mapsto (-m,b) \in \R^2.
\]
The minus sign here is used to preserve incidence relations. In particular, given $L \in \mathrm{im}\mathbf{D}$, then $x\in L \iff \mathbf{D}^\ast(L) \in \mathbf{D}(x)$, though we won't need this here. In the plane, this relation is referred to as \textit{point-line duality}. In higher dimensions there is an analogous notion of \textit{point-hyperplane duality}, with the more general map $\mathbf{D}:\R^n \to \mathcal A(n,n-1)$ given by 
\[
(x_1,\dots, x_n) \mapsto \left\{(y_1,\dots, y_{n-1}, \sum_{i=1}^n x_iy_i + x_n) : (y_1,\dots, y_{n-1}) \in \R^{n-1}\right\}.
\]
There is a similar higher dimensional analogue of $\mathbf{D}^\ast: \mathrm{im}\,\mathbf{D}\to \R^n$, see \cite{dabrowski2021integrability}. Note that this map is dimension preserving. For our purposes, this will allow us to think about our hyperplanes in Proposition \ref{prop:appendixstatement} as points in $\R^n$.

\begin{rem}
    The application of point-hyperplane duality to the study of Furstenberg sets has been utilized quite a bit. See for instance \cite{osw, liliu}.
\end{rem}

We will also need Marstrand's projection theorem which we state now. Recall that given $U \in \mathcal G(n,k)$, we denote the orthogonal projection onto $U$ as $\pi_U$.

\begin{thm}[Marstrand's Projection Theorem] \label{thm:marstrand}
    Let $A\subset \R^n$ be a Borel set. Then, for almost every $U\in \mathcal G(n,k)$, 
    \[
    \dim \pi_U(A) = \min\{\dim A, k\}.
    \]
    Here, almost every is with respect to the natural Haar measure $\gamma_{n,k}$ on $\mathcal G(n,k).$
\end{thm}

We briefly remark on Marstrand's projection theorem before moving on. Marstrand proved the above projection theorem in '53 for projections onto lines in the plane, and Mattila generalized the result for projections onto $k$-dimensional subspaces in '75 (see \cite[Chapter 5]{mattilaBook} for more background). Notice that the equality in Marstrand's projection theorem is the best one should expect for a generic $k$-plane, as (naturally) $\dim \pi_U(A) \leq \min\{\dim A, k\}$ for all $U \in \mathcal G(n,k)$. This follows as orthogonal projection is a Lipschitz map. A natural question comes from investigating how often the ``size'' of the projection can be ``smaller.'' E.g., given $s\in (0,\min\{\dim A,k\}]$, consider the set 
\[
E(A) := \{U\in \mathcal G(n,k) : \dim \pi_U(A) < s\}.
\]
Marstrand's projection theorem says that the above set has zero-measure with respect to $\gamma_{n,k}$, but how large can the set of such ``exceptional'' subspaces be, say, in terms of Hausdorff dimension? Note that the set $E(A)$ is referred to as an \textit{exceptional set}, and upper bounds on $\dim E(A)$ are known as \textit{exceptional set estimates}. The investigation of such \textit{exceptional set estimates} has proven quite fruitful in recent years. It may be worth noting that there are close connections between Furstenberg set estimates and exceptional set estimates (especially in the discretized setting), though we don't go into this further here. For more, see e.g. \cite{renwang}.

In any case, point-hyperplane duality and Marstrand's projection theorem are the main tools we need to prove Proposition \ref{prop:appendixstatement}, which we now prove.

\begin{proof}[Proof of Proposition \ref{prop:appendixstatement}]
Let $F \subset \R^n$ be an $(s,t;n-1)$-Furstenberg set with associated set of hyperplanes given by $\mathcal P \subset \mathcal A(n,n-1)$. Without loss of generality, we can assume $\dim \mathcal P = t$. Then, by point-hyperplane duality, we can think about our set of hyperplanes as points in $\R^n$. In particular, $\mathbf{D}^\ast(\mathcal P)$ is a set of points in $\R^n$ with $\dim \mathbf{D}^\ast(\mathcal P) = \dim \mathcal P =t$. Then, by Marstrand's projection theorem (Theorem \ref{thm:marstrand}), it follows that for $\gamma_{n,n-1}$-almost every subspace $U\in \mathcal G(n,n-1)$,
\[
\dim \pi_U(\mathbf{D}^\ast (\mathcal P)) = \min\{\dim \mathbf{D}^\ast(\mathcal P), n-1\} = \min\{t,n-1\}.
\]
Hence, fix one such $U$ such that $\dim \pi_U (\mathbf{D}^\ast(\mathcal P)) = \min\{t,n-1\}.$ 
From this orthogonal projection, we can deduce that there exists a projective transformation $\phi^\ast$ mapping $U$ to hyperplane at infinity (see \cite{mattilaBook}). More precisely, there exists a dimension preserving projective map $\phi^\ast : \R^n \to \mathcal A(n,n-1)$ mapping 
\[
\mathbf{P} = \pi_U(\mathbf{D}^\ast(\mathcal P)) \ni x\mapsto V(x) \in \mathcal G(n,n-1)
\]
such that 
\[
\dim \phi^\ast(\mathbf{P}) = \min\{t,n-1\}.
\]
For more details on this projective transformation and the notation $V(x)$, see \cite[Remark 4.13]{osw}. 
This induces a projective transformation $\phi: \R^n \to \R^n$ such that $\phi(F)$ is an $(s,t;n-1)$-Furstenberg set with associated set of affine hyperplanes $\mathbf{P}:= \mathbf{D}(\phi^\ast(\mathbf{D}^\ast(\mathcal{P}))).$ 

The fact that $\phi(F)$ is an $(s,\min\{t,n-1\};n-1)$-\textit{spread} Furstenberg set follows from the fact that $\mathbf{P}$ is ``spread out.'' In particular, let $p : \mathcal A(n,n-1) \to \mathcal G(n,n-1)$ be the map that sends the affine hyperplane $W = U + a_W$ to $U$ (where $U \in \mathcal G(n,n-1)$). Then, by our application of Marstrand's projection theorem (and our construction of $\phi^\ast$. sending $U$ to infinity), it follows that 
\[
\dim p(\mathbf{P}) = \min\{t,n-1\}.
\]
Hence, it follows that $\phi(F)$ is an $(s,\min\{t,n-1\};n-1)$-spread Furstenberg set with associated set of \textit{subspaces} $p(\mathbf{P})$.
\end{proof}

We conclude with an example to see how this works out geometrically. Let $\mathcal P$ be a 1-dimensional set of 
horizontal lines in the plane, and let $F = \bigcup_{P\in \mathcal P} \ell_P$
 where for all $P\in \mathcal P$, $\ell_P$ is a unit line segment contained in $P.$ Clearly, $F$ is a $(1,1;1)$-Furstenberg set, and the $(s,t;1)$-Furstenberg set estimate of Orponen--Shmerkin \cite{orponenshmerkinabc} and Ren--Wang \cite{renwang} implies that $\dim F = 2$. We claim that $F$ is also a $(1,1;1)$-spread Furstenberg set, though this isn't immediately clear to see from the geometry. In particular, it would be natural to associate each line in $\mathcal P$ to the subspace given by the horizontal line through the origin (this is the map $p: \mathcal A(2,1) \to \mathcal G(2,1)$ that we used in the proof). However, $p(\mathcal{P})$ is just the $x$-axis, so clearly this natural identification using $p$ isn't enough to capture how large the set of $\mathcal P$ is as a set of affine lines. This is where point-line duality and Marstrand's projection theorem can come into play.

Consider the set of points given by $\mathbf{D}^\ast (\mathcal P)$ where again $\mathbf{D}^\ast(\{y = mx + b\}) = (-m,b)$. Since these are horizontal lines, notice that $\mathbf{D}^\ast(\mathcal P)$ is contained on the $x$-axis. Furthermore, notice that the map $p$ described earlier is precisely the same as projecting onto the $y$-axis (and in particular, $\dim p(\mathbf{D}^\ast(\mathcal P)) = 0$). However, projecting onto another direction (say, projecting onto the line with slope 1 through the origin) will be a bi-Lipschitz map, and in particular such a map will preserve the dimension of $\mathbf{D}^\ast(\mathcal P)$. Call the line with slope 1 through the origin $U$. This is precisely where/how we used Marstrand's projection theorem in the proof of Proposition \ref{prop:appendixstatement}. Furthermore, this makes it quite easy to visualize the projective transformation $\phi^\ast$; $\phi^\ast$ is a projective map (technically given by conjugation of matrices in $O(n)$) that sends $U$ to the $y$-axis. This is more or less simply changing the parameterization of the set of affine lines in $\R^2$, but now notice that (by Marstrand's projection theorem and our choice of projective tranformation $\phi^\ast$),
\[
\dim p(\underbrace{\mathbf{D}(\phi^\ast(\mathbf{D}^\ast (\mathcal P)))}_{\mathbf{P}}) = \min\{\dim \mathcal P, 1\} = 1.
\]
Furthermore, $p(\mathbf{P})$ is the associated 1-dimensional set of \textit{subspaces} that shows that $\phi(F)$ is a $(1,1;1)$-\textit{spread} Furstenberg set.

\end{document}